\documentclass[runningheads, a4paper]{llncs}

\usepackage{amsmath,amsfonts,amssymb}
\usepackage{setspace}
\usepackage{fancyhdr}
\usepackage{lastpage}
\usepackage{extramarks}
\usepackage{chngpage}
\usepackage{tikz}
\usepackage{soul,color}
\usepackage{enumerate}
\usepackage{graphicx,float,wrapfig}

\newtheorem{restate}{Theorem}
\newcommand{\card}[1]{\ensuremath{\left|#1\right|}}
\newcommand{\bd}[1]{{\boldsymbol #1}}

\begin{document}

\title{On a Conjecture of Butler and Graham}

\author{
Tengyu Ma
\inst{1} \and Xiaoming Sun\inst{2} \and Huacheng Yu\inst{1} }

\institute{Institute for Interdisciplinary Information Sciences,
Tsinghua University
\and
Institute of Computing Technology, Chinese Academy of Sciences
}

\date{}

\maketitle

\begin{abstract}
Motivated by a hat guessing problem proposed by Iwasawa~\cite{Iwasawa10}, Butler and Graham~\cite{Butler11} made the following conjecture on the existence of certain way of marking the {\em coordinate lines} in $[k]^n$: there exists a way to mark one point on each {\em coordinate line} in $[k]^n$, so that every point in $[k]^n$ is marked exactly $a$ or $b$ times as long as the parameters $(a,b,n,k)$ satisfies that there are non-negative integers $s$ and $t$ such that $s+t = k^n$ and $as+bt = nk^{n-1}$. In this paper we prove this conjecture for any prime number $k$. Moreover, we prove the conjecture for the case when $a=0$ for general $k$.

%
%


{\bf MSC classes:} 00A08 97A20 94B05
\end{abstract}

\setcounter{footnote}{0}

\section{Introduction}

In~\cite{Butler11} Butler and Graham considered the problem of the existence of certain way of marking coordinate lines in $[k]^n$.
A \textit{coordinate line} in $[k]^n$ is the set of $k$ points in which all but one coordinate are fixed and the unfixed coordinate varies over all possibilities.
\textit{Marking} a line\footnote{for convenience, we use line to indicate coordinate line throughout this paper.} means designating a point on that line.
They conjectured that
\begin{conjecture}[Butler, Graham~\cite{Butler11}]\label{conj}
There is a marking of the lines in $[k]^n$ so that on each line exactly one
point is marked and each point is marked either $a$ or $b$ times if and only if there are nonnegative integers $s$ and $t$ satisfying the linear equations $s + t = k^n$ and $as + bt = nk^{n-1}$.
\end{conjecture}

The ``only if'' part of the conjecture is straightforward, leaving to be crucial the construction of a marking of lines in $[k]^n$ with the desired properties. Buhler, Butler, Graham, and Tressler~\cite{Buhler10} proved the conjecture for $k = 2$. Butler and Graham~\cite{Butler11} proved the conjecture when $n \le 5$. The main contributions of this paper are i) we prove all the cases when $k$ is an odd prime; ii) we prove the case when $a = 0$ (without any assumption on $k$).

\begin{theorem}\label{primeTheorem}
For any prime $k$ and $0\leq a < b\leq n$, there exists a marking of lines in $[k]^n$ so that each point is marked either $a$ times or $b$ times if and only if there are nonnegative integers $s$ and $t$ so that $s+t = k^n$ and $as+bt = nk^{n-1}$.
\end{theorem}

\begin{theorem}\label{a0Theorem}
For $0< b\leq n$, there exists a marking of lines in $[k]^n$ so that each point is either unmarked or marked $b$ times if and only if there are nonnegative integers $s$ and $t$ so that $s+t = k^n$ and $bt = nk^{n-1}$.
\end{theorem}

As in~\cite{Butler11}, we use the notation $[a,b]_k^n$ as a shorthand for a realization of a marking of the lines in $[k]^n$ where each point is marked either $a$ times or $b$ times. Then Theorem~\ref{primeTheorem} provides a sufficient and necessary condition for the existence of $[a,b]_k^n$ for any prime $k$, and Theorem~\ref{a0Theorem} considers the existence of $[a,b]_k^n$ when $a = 0$.

In the proof of Theorem~\ref{primeTheorem} we reduce the existence of $[a,b]_k^n$ to the existence of $[a-1,b-1]_k^{n-k}$ when $a>0$ and $b\le n- k+ 1$~(see Proposition~\ref{induction} in Section~\ref{section-prime}). It turns out that the most complicated part of this inductive argument is the construction of the base cases, that is, $[0,b]_k^n$ and $[a,n-t]_k^n$~(where $t < k$). We provide two theorems~(Theorem~\ref{a,n-t} and \ref{0b}) giving a direct realizations for these two kinds of markings. The proofs of these theorems use similar approaches, though different in many details, that we partition the whole grid $[k]^n$ according to certain number theory based characteristic function, which has a nice symmetrical property: Informally speaking, for any fixed $s$, any point $\vec{x}\in [k]^n$, and any value $\bd{v}$ in the range of this function, there exists a unique direction along which by moving from $\bd{x}$ with distance $s$, we can reach a point with value $\bd{v}$. By a sophisticated utilizing this property, we accomplish the design of markings for these two base cases. Furthermore, we prove the case $[0,b]_k^n$ for general $k$ by generalizing the characteristic function in a delicate way.

\vskip10pt

\noindent {\bf Related Work}

The motivation of investigating this marking line problem is to reformulate and solve a hat guessing question proposed by Iwasawa~\cite{Iwasawa10}.
In that game there are several players sitting around a table, each of which is assigned a hat with one of $k$ colors. Each player can see all the colors of others' hat but his/her own. The players try to coordinate a strategy before the game starts, and guess the colors of their own hats simultaneously and independently after the hats are placed on their heads. Their goal is to design a strategy that guarantees exactly either $a$ or $b$ correct guesses. For example, one special case is that either everybody guesses the color correctly or nobody guesses correctly, i.e. $a=0$ and $b=n$.

Several variations of hat guessing game have been considered in the literature. Ebert~\cite{Ebert98} considered the model that players are allowed to answer ``unknown". He showed that in this model there is a {\em perfect strategy} for players when $n$ is of the form $2^m-1$. Lenstra and Seroussi~\cite{Lenstra05} studied the case that $n$ is not of such form. Butler, Hajiaghayi, Kleinberg and Leighton~\cite{Butler09} considered the worst case that each player can see only part of the others' hats with respect to a sight graph. Feige~\cite{Feige10} investigated the average case with a sight graph. Peterson and Stinson~\cite{PS10} investigated the case that each player can see hats in front of him and they guess one by
one. Recently Ma, Sun and Yu \cite{MSY11} proposed a variation which allow to answer ``unknown" and require at least $k$ correct guesses for winning condition.


\vskip10pt
\noindent{\bf Notations and Preliminaries}


\noindent
$[k]=\{1,2,\ldots,k\}$, $[k]^n=\underbrace{[k]\times \cdots \times[k]}_{n}$. $[a,b]_k^n$: marking of
lines in $[k]^n$ in which each point is marked either $a$ times or $b$ times. We assume that $a< b$ as well.
Throughout the paper we always use boldface type letters for vectors and vector-valued functions.
For a vector $\bd{x}=(x_1,\ldots,x_n)$, define the $k$-modulo parity function $\oplus(\bd{x})=x_1+x_2+\cdots+x_n\mod k$.
We denote by $\bd{x}_{-i}$ the line $(x_1,\dots,x_{i-1},*, x_{i+1},\dots,x_n)$, i.e.
$$\bd{x}_{-i}=\{(x_1,\ldots,x_{i-1},y,x_{i+1},\ldots,x_n)~|~y\in [k]\}.$$

The necessary condition in the conjecture is straightforward.
\begin{proposition}[\cite{Buhler10}] \label{necessary}
If we have $[a,b]_k^n$, then the following equations system has nonnegative integer solution.
$$\left\{
\begin{array}{l}
  s+t=k^n, \\
  as+bt=k^{n-1}n.
\end{array}
\right.
$$
More specifically, $s = \frac{k^{n-1}(kb-n)}{b-a}$ is the number of points that are marked $a$ times, and $t = \frac{k^{n-1}(n-ka)}{b-a}$ is the number of points that are marked $b$ times.
\end{proposition}

\vskip10pt

The rest of the paper is organized as follows: In Section~\ref{section-prime} we show that the necessary condition is sufficient when the number of colors $k$ is a prime. Section~\ref{section-general} considers the case for general $k$ when $a=0$. Finally we conclude the paper in Section~\ref{section-conclusion} with some open problems.

\section{$[a,b]_k^n$ for Prime $k$}\label{section-prime}

In this section we prove the conjecture when $k$ is an odd prime number (the case $k=2$ has been proved by Buhler et al.~\cite{Buhler10}). The first step is to reduce $[a,b]_k^n$ to $[a-1,b-1]_{k}^{n-k}$ as in Butler and Graham \cite{Butler11}.

\begin{proposition}[\cite{Butler11}] \label{induction}
Given $[a,b]_{k}^{n}$, we have $[a+1,b+1]_k^{n+k}$.
\end{proposition}
By repeatedly using this Proposition, we can reduce the problem $[a,b]_k^n$ to two possible cases: (1) $[a',b']_k^{n'}$, where $b'>n'-k$; (2) $[0,b']_k^{n'}$~(recall that $a<b$). During this procedure, the divisibility is unchanged (see the proof of Theorem~\ref{primeTheorem}). The following two theorems~(Theorem~\ref{a,n-t} and Theorem~\ref{0b}) give the constructions of two base cases, respectively.


\begin{theorem}\label{a,n-t}
If $k$ is a prime and $ 0 \le t \le k-1$, $1\le a< n-t$, and $(\frac{n-ka}{n-t-a})k^{n-1}$ is a nonnegative integer, then we have $[a,n-t]_k^n$.
\end{theorem}

\begin{proof}
Firstly we do some elementary number theory substitution to make the parameters more manageable. Suppose that $(\frac{n-ka}{n-t-a})k^{n-1}$ is a nonnegative integer. Since $k$ is a prime, there exists $m,r\in \mathbb{N}$
so that $n-t-a = k^m r$, where $m \le n-1$, $r|n-ka$, and $n\ge ka$.
($n-t\ge a$ implicitly holds.) Observe 
that $r|n-ka$ and $r|n-t-a$ implies that $r|(k-1)a-t$. Let $(k-1)a-t = ra'$, where
$$
a'=\frac{(k-1)a-t}{r}=\frac{(ka-a-t)k^m}{n-t-a}\leq \frac{(n-a-t)k^m}{n-t-a}=k^m.
$$
Thus, $n = t+a+k^m r$ and $(k-1)a = t+ra'$, where $a'$, $m$ and $r$ are all nonnegative integers.

Now we construct a marking of lines in $[k]^{n}$ so that each point is marked either $(n-t)$ times or $a$ times. For convenience, we partition the $n = a+t+k^m r$ dimensions into three groups, each of which contains $k^m r$, $t$, and $a$ coordinates respectively\footnote{When $t=0$, then there is only two groups, the proof still holds.}, and represent each point in $[k]^{n}$ by $(\bd{x},\bd{y},\bd{z})$ where $\bd{x} \in [k]^{k^m r}$, $\bd{y} \in [k]^t$, and $\bd{z} \in [k]^a$. Furthermore, we index $\bd{x}$ by a pair $(i,j)\in \mathbb{Z}_k^m\times [r]$, i.e. $x_{i,j}\in [k]$ are the coordinates of $\bd{x}$~(where $i \in \mathbb{Z}_k^m, j\in [r]$)\footnote{Since we need to take some ring operations on the index $i$, we index $i$ by $\mathbb{Z}_k^m$ here instead of $[k]^m$.}. For a point $(\bd{x},\bd{y},\bd{z})\in [k]^n$ and $i\in \mathbb{Z}_k^m, j\in[r]$, denote by $(\bd{x}_{-(i,j)},\bd{y},\bd{z})$ the lines of $[k]^{n}$ for which all the coordinates are fixed except the coordinate $(i,j)$ of $\bd{x}$, i.e.
\[(\bd{x}_{-(i,j)},\bd{y},\bd{z})=\{(\tilde{\bd{x}},\tilde{\bd{y}},\tilde{\bd{z}})\in [k]^n~|~\tilde{\bd{y}}=\bd{y},\tilde{\bd{z}}=\bd{ z},\forall~ (i',j')\neq (i,j)~\tilde{x}_{i',j'}=x_{i',j'}, \tilde{x}_{i,j}\in[k]\}
\]
Similarly $(\bd{ x},\bd{y}_{-i},\bd{z})~(i\in [t])$ is a line of $[k]^{n}$ which the $i$-th coordinate of $\bd{y}$ is unfixed, $(\bd{x},\bd{y},\bd{z}_{-i})~(i\in [a])$ is a line which the $i$-th coordinate of $\bd{z}$ is unfixed.

\vskip10pt
For each $\bd{x}\in [k]^{k^m r}$, define the {\em characteristic function} $\bd{q} : [k]^{k^m r} \rightarrow [k]^m$ as follows:
\begin{equation}
    \bd{q}(\bd{x}) = \sum_{\bd{i}=(i_1,\ldots,i_m)\in \mathbb{Z}_k^m} \left(\sum_{j=1}^r x_{i,j}\right)\cdot \bd{i} \label{charFun}
\end{equation}
where $\bd{i}$ is
$k$-based representation of $i$ in $\mathbb{Z}^m$, and the operations $+$ and $\cdot$ are over $\mathbb{Z}_k$. According to the characteristic value $\bd{q}(\bd{x})$, we can group the points in $[k]^{n}$ into equivalence classes. Specifically, for any $\bd{w} \in [k]^m$, let $Q(\bd{w})$ be the collection of points in $[k]^n$ which have characteristic value $\bd{w}$, i.e.
\[
    Q(\bd{w}) = \left\{ (\bd{x},\bd{y},\bd{z}) \in [k]^n \mid \bd{q}(\bd{x}) = \bd{w} \right\}.
\]

Arbitrarily choose $a'$ different values $\bd{w}_1,\dots,\bd{w}_{a'}$ from $[k]^m$, for example the first $a'$ elements in the lexicographical order (recall that $0 \le a' \le k^m$), and let $M$ be the collections of points which have one of these $a'$ values as characteristic value and $0$ as $k$-modulo parity, i.e.,
\[
    M = \left( Q(\bd{w}_1)\cup\dots\cup Q(\bd{w}_{a'})\right) \cap \{(\bd{x},\bd{y},\bd{z}) \in [k]^n \mid \oplus(\bd{x},\bd{y},\bd{z}) = 0\}.
\]
We arbitrarily partition the set $[a']\times [r]$~(recall that $[a']\times [r] \subset [k^m]\times [r]$ is the indices set of $\bd{x}$) into $(k-1)$ subsets $L_1\sqcup L_2\sqcup \cdots \sqcup L_{k-1}$ with the requirement that the cardinality $\card{L_1} = \dots = \card{L_t} = a-1$ and $\card{L_{t+1}} = \dots = \card{L_{k-1}} = a$. (Notice that $t(a-1)+(k-1-t)a = (k-1)a-t = a'r$.)

\vskip10pt

Based on these preparations, now we give the construction of the desired marking of $[k]^n$. There are three different types of lines: $(\bd{x}_{-(i,j)}, \bd{y},\bd{z})$, $(\bd{x},\bd{y}_{-i},\bd{z})$, and $(\bd{x},\bd{y},\bd{z}_{-i})$,
we mark them as follows:

\begin{enumerate}
    \item for the line $(\bd{x}_{-(i,j)}, \bd{y},\bd{z})$, there are two sub-cases:
        \begin{itemize}
            \item $(\bd{x}_{-(i,j)}, \bd{y},\bd{z})\cap M\neq \emptyset$, i.e. on line $(\bd{x}_{-(i,j)}, \bd{y},\bd{z})$ there exists some point belongs to set $M$. (by the condition that $\oplus(\bd{x},\bd{y},\bd{z})=0$, this point is unique, if exists). Suppose the point is $(\tilde{\bd{x}},\bd{y},\bd{z})\in (\bd{x}_{-(i,j)}, \bd{y},\bd{z})\cap M$, and suppose that $( \tilde{\bd{x}},\bd{y},\bd{z}) \in Q(\bd{w}_{i_0})$ for some $i_0 \in [a']$, and $(i_0,j) \in L_s$ for some $s\in [k-1]$ (recall that $L_1,\dots, L_{k-1}$ is a partition of $[a']\times [r]$). Then we mark the point $(\tilde{\bd{x}}+s\cdot \bd{e}_{i,j},\bd{y},\bd{z})\in (\bd{x}_{-(i,j)}, \bd{y},\bd{z})$, where $\bd{e}_{i,j}=(0,\ldots,0,1,0,\ldots,0)$ is the unit vector in $[k]^{k^m r}$ for which only $x_{i,j} = 1$, and all other coordinates equal $0$. The addition and multiplication are over $\mathbb{Z}_k$. This is also the unique point on this line which has $\oplus(\cdot)=s$;

        \vskip5pt
            \item otherwise $(\bd{x}_{-(i,j)}, \bd{y},\bd{z})\cap M=\emptyset$, mark the unique point $(\tilde{\bd{x}},\bd{y},\bd{z})$ on the line such that $\oplus(\tilde{\bd{x}},\bd{y},\bd{z}) = 0$.
        \end{itemize}

        \vskip5pt

    \item for line $(\bd{x},\bd{y}_{-i},\bd{z})~(i\in [t])$, mark the unique point $(\bd{x},\tilde{\bd{y}},\bd{z})$ which satisfies $\oplus(\bd{x},\tilde{\bd{y}},\bd{z}) = i$. (recall that $1\leq i\leq t\leq k-1$)

        \vskip5pt
    \item for line $(\bd{x},\bd{y},\bd{z}_{-i})~(i \in [a])$, mark the point $(\bd{x},\bd{y},\tilde{\bd{z}})$ which satisfies $\oplus(\bd{x},\bd{y}, \tilde{\bd{z}}) = 0$.
\end{enumerate}

\vskip5pt
 We claim that the construction above is indeed a $[a,n-t]_{k}^n$. We need to check that each point $(\bd{x},\bd{y},\bd{z})$ is marked either $a$ times or $(n-t)$ times.

\begin{enumerate}
    \item if $\oplus(\bd{x},\bd{y},\bd{z}) = 0$, then
         for each line $(\bd{x},\bd{y}_{-i},\bd{z})$, we never mark the point $(\bd{x},\bd{y},\bd{z})$ (recall that we mark some point which has $\oplus(\cdot)=i \neq 0$). On the contrary, for each line $(\bd{x},\bd{y},\bd{z}_{-i})$, we always mark $(\bd{x},\bd{y},\bd{z})$. For lines of the form $(\bd{x}_{-(i,j)},\bd{y},\bd{z})$, there are two sub-cases:

     \vskip5pt
        \begin{itemize}
            \item if $(\bd{x},\bd{y},\bd{z}) \in M$, then on the line $(\bd{x}_{-(i,j)},\bd{y},\bd{z})$ we mark the point $(\bd{x}+s \cdot \bd{e}_{i,j},\bd{y},\bd{z})$ for some $s>0$ by the construction, which is not $(\bd{x},\bd{y},\bd{z})$. Thus in this case, $(\bd{x},\bd{y},\bd{z})$ is marked $0$, $0$, and $a$ times in $\bd{x}$, $\bd{y}$ and $\bd{z}$'s directions respectively, hence $a$ times in total;

        \vskip5pt
            \item if $(\bd{x},\bd{y},\bd{z}) \not \in M$, then on the line $(\bd{x}_{-(i,j)},\bd{y},\bd{z})$ there is no point in $M$. Thus  we marked the point with $\oplus(\cdot)=0$, which is exactly point $(\bd{x},\bd{y},\bd{z})$ itself. In this case, $(\bd{x},\bd{y},\bd{z})$ is marked $k^mr$, $0$, and $a$ times in three groups of directions respectively, and $k^mr+0+a=n-t$ times in total.
        \end{itemize}
        Therefore,  $(\bd{x},\bd{y},\bd{z})$ is marked either $a$ or $(n-t)$ times.

    \vskip5pt

    \item if $\oplus(\bd{x},\bd{y},\bd{z}) = s$ for some $1\le s \le t$.
        Among lines $(\bd{x},\bd{y}_{-i},\bd{z})~(1\le i \le t)$, only on the line $(\bd{x},\bd{y}_{-s},\bd{z})$ we marked point $(\bd{x},\bd{y},\bd{z})$. On each line $(\bd{x},\bd{y},\bd{z}_{-i})~(i \in [a])$, we never mark $(\bd{x},\bd{y},\bd{z})$.

        By the definition of the characteristic function, $\bd{q}(\bd{x}-s\cdot \bd{e}_{i,j})$ are different for different $i$'s (here we use the fact that $k$ is a prime, hence $s$ is coprime to $k$ and has an inverse in $\mathbb{Z}_k$). Therefore there are exactly $a' \times r$ different pairs of $(i,j)$ such that line $(\bd{x}_{-(i,j)},\bd{y},\bd{z})$ contains a point $(\bd{x}-s \cdot \bd{e}_{i,j}, \bd{y}, \bd{z})\in M$. On exactly $\card{L_s} =a-1$ lines of these $a'\times r$ lines, point $(\bd{x},\bd{y},\bd{z})$ is marked. On all other lines, there is no point belong to $M$, thus we only mark the point with $\oplus(\cdot)=0$.
        Thus the point $(\bd{x},\bd{y},\bd{z})$ is marked $(a-1)$, $1$ and $0$ times in $\bd{x}$, $\bd{y}$, and $\bd{z}$'s directions, respectively, and in total $a$ times.

    \vskip5pt

    \item if $\oplus(\bd{x},\bd{y},\bd{z}) = s$ for some $s > t$. It is similar to the case above, except that we never mark point $(\bd{x},\bd{y},\bd{z})$ on lines of form $(\bd{x},\bd{y}_{-i},\bd{z})$, and on $\card{L_s} = a$ lines of form $(\bd{x}_{-(i,j)},\bd{y},\bd{z})$ we mark $(\bd{x},\bd{y},\bd{z})$. Therefore in this case $(\bd{x},\bd{y},\bd{z})$  is also marked exactly $a$ times.

\end{enumerate}
\qed
\end{proof}

\begin{theorem}\label{0b}
If $k$ is a prime, $0<b\leq n$, and $\left(\frac{kb-n}{b}\right)k^{n-1}$ is an integer, then we have $[0,b]_k^n$.
\end{theorem}

\begin{proof}
Suppose $\gcd(b,n)=r$, since $k$ is a prime number and $\left(\frac{kb-n}{b}\right)k^{n-1}\in \mathbb{Z}$, we have that $b=rk^m$, $n=rn'$ for some non-negative integer $m$ and $n'$. By the following Proposition, it suffices to prove the $r=1$ case.
\begin{proposition}[\cite{Butler11}]\label{mul}
	Given $[0,b]_k^n$, then for every $r$, we have $[0,br]_k^{rn}$.
\end{proposition}

Since $b \leq n \leq kb$, let $n=tb+h$, where $1 \leq t< k$ and $0 \leq h\leq b$. We partition the $n$ coordinates into three groups, each of which contains $k^m$, $(t-1)b$, and $h$ coordinates (note that now $b = k^m$), respectively\footnote{If $t=1$ or $h= 0$, one of the group probably vanishes, while the proof still holds.}. We represent each point in $[k]^n$ by $(\bd{x}, \bd{y}, \bd{z})$, where $\bd{x} \in [k]^{k^m}$, $\bd{y} \in [k]^{(t-1)b}$, and $\bd{z}\in [k]^h$. Notations $\left(\bd{x}_{-i},\bd{y},\bd{z}\right)$, $(\bd{x},\bd{y}_{-(i,j)},\bd{z})$ and $\left(\bd{x},\bd{y},\bd{z}_{-i}\right)$ are similarly defined as in the previous proof (note that for $(\bd{x},\bd{y}_{-(i,j)},\bd{z})$, the indexes $i\in [t-1],j\in [b]$).

Similarly we define the characteristic function $\bd{q} : [k]^{k^m } \rightarrow [k]^m$ as follows:
   \[
   	\bd{q}(\bd{x}) = \sum_{\bd{i} = (i_1,\dots, i_m) \in \mathbb{Z}_k^m} x_{i}\cdot\bd{i}.
   \]
   Let $Q(\bd{w}) = \{(\bd{x},\bd{y},\bd{z})\in [k]^n \mid \bd{q}(\bd{x}) = \bd{w}\}$ as usual, and the notation of $\bd{i}$ and $\cdot, +$ are the same as in the proof of Theorem \ref{a,n-t}.
   We arbitrarily choose $h$ different values $\bd{w_1},\bd{w_2},\ldots,\bd{w_h}$ from $[k]^m$, and let
   \[
   	M = \left( Q(\bd{w}_1)\cup\dots\cup Q(\bd{w}_{h})\right) \cap \{(\bd{x},\bd{y},\bd{z}) \in [k]^n \mid \oplus(\bd{x},\bd{y},\bd{z}) = 0\}.
   \]
   Now we describe the marking of the line (for some fixed point $(\bd{x},\bd{y},\bd{z})$):

   \begin{enumerate}
      \item
      	 For the line $(\bd{x_{-i}},\bd{y}, \bd{z})$, if there exists a point $(\tilde{\bd{x}}, \bd{y},\bd{z})\in M$ on it, then we mark this point. Otherwise we mark the unique point $(\tilde{\bd{x}}, \bd{y},\bd{z})$ with $k$-modulo parity $\oplus(\tilde{\bd{x}}, \bd{y},\bd{z}) = 1$.

    \vskip5pt
      \item
      	 For line $(\bd{x},\bd{y}_{-(i,j)}, \bd{z})~(i\in [t-1],j\in [b]$), we mark the point $(\bd{x}, \tilde{\bd{y}},\bd{z})$ on the line with parity $\oplus(\bd{x}, \tilde{\bd{y}},\bd{z}) = i+1$.

    \vskip5pt
      \item
      	 On line $(\bd{x}, \bd{y}, \bd{z}_{-i})~( i\in [h])$, we mark the unique point $(\bd{x}, \bd{y},\tilde{\bd{z}})$ with parity $\oplus(\bd{x}, \bd{y},\tilde{\bd{z}}) = 1$.
   \end{enumerate}

\vskip5pt
 We claim that the construction above is a $[0,b]_{k}^n$. We need to verify that each point $(\bd{x}, \bd{y}, \bd{z})$ is marked either $b = k^m$ times or unmarked. There are three cases:
   \begin{enumerate}
      \item
		$\oplus(\bd{x}, \bd{y}, \bd{z})=0$.
	 \begin{itemize}
	 	\item if $(\bd{x}, \bd{y}, \bd{z}) \in M$. The point is only marked by lines of the form $(\bd{x}_{-i}, \bd{y}, \bd{z})$. There are $b$ such lines;

	 	\item otherwise, the point is never marked.
	 \end{itemize}

    \vskip5pt
      \item
	 $\oplus(\bd{x}, \bd{y}, \bd{z})=1$. $\bd{q}(\bd{x}-\bd{e}_i)$ are pairwise different, thus on exactly $(b-h)$ lines $(\bd{x}_{-i}, \bd{y},\bd{z})$ there is no point in $M$. On these lines $(\bd{x}, \bd{y}, \bd{z})$ is marked. All the lines of form $(\bd{x}, \bd{y}_{-i}, \bd{z})$ will not mark the point $(\bd{x}, \bd{y}, \bd{z})$, and all $h$ lines of form $(\bd{x}, \bd{y}, \bd{z}_{-i})$ will mark this point. Thus  $(\bd{x}, \bd{y}, \bd{z})$ is marked $b$ times in total.

    \vskip5pt
      \item
         $2 \le \oplus(\bd{x}, \bd{y}, \bd{z}) \le t$. The point is marked on all the line of the form $(\bd{x},\bd{y}_{-(i,j)}, \bd{z})$ where $i = \oplus(\bd{x}, \bd{y}, \bd{z}) - 1$ and $j\in [b]$. There are $b$ such lines.

    \vskip5pt
      \item
	 $\oplus(\bd{x}, \bd{y}, \bd{z}) > t$. The point is not marked.
   \end{enumerate}
\qed
\end{proof}

Now we are ready to present our main theorem.

\begin{restate}[Restated]
For prime $k$ and $0\leq a < b\leq n$, there exists a marking of lines in $[k]^n$ so that each point is marked either $a$ times or $b$ times if and only if there are nonnegative integers $s$ and $t$ so that $s+t = k^n$ and $as+bt = nk^{n-1}$.
\end{restate}
\begin{proof}
The necessary part of the theorem is trivial and has been shown in the preliminary section of the introduction. The proof of sufficiency is an induction on $n$ by essentially using Theorem~\ref{a,n-t} and \ref{0b} as base steps.

The $n = 1$ case is obvious. Assume that for any $n \le m-1$, the theorem holds. Now we prove the theorem for $n = m$. There are three possible occasions:
\begin{enumerate}
\item If $a = 0$, then by Theorem~\ref{0b} the theorem holds.

\item If $b > m-k$, by Theorem~\ref{a,n-t} the theorem holds.

\item If $a > 0$ and $b \le m-k$. Assume
$$s=\frac{k^{m-1}(kb-m)}{b-a}, \ \ t=\frac{k^{m-1}(m-ka)}{b-a}$$ are both nonnegative integers.
We claim that both
$$s'=\frac{k^{m-k-1}[k(b-1)-(m-k)]}{b-a}\ \makebox{ and }~t'=\frac{k^{m-k-1}[(m-k)-k(a-1)]}{b-a}$$ are nonnegative integers.
Suppose that $b-a = rk^d$, where $\gcd(r,k) = 1$. Thus we have that $r |m-ka$, since $r | k^{m-1}(m-ka)$ and $\gcd(r,k^{m-1}) = 1$. Since
$$k^d \le b- a \le m-k-1,$$
we have that $d\leq m-k-1$ and $k^d | k^{m-k-1}$. Then $rk^d | k^{m-k-1}(m-ka)$, that is
$$
b-a | k^{m-k-1}\left((m-k)-k(a-1)\right).
$$
Similar argument shows that
$$
b-a| k^{m-k-1}\left(k(b-1)-(m-k)\right).$$
Since $a>0$ and $b\leq m-k$, we still have $0\leq a-1\leq b-1\leq m-k$.

By invoking inductive hypothesis, we have that $[a-1,b-1]_k^{m-k}$ exists. Therefore, by applying Proposition~\ref{induction} we have that $[a,b]_k^m$ exists.
\end{enumerate}
\qed
\end{proof}

\section{$[0,b]_{k}^n$ for General $k$}\label{section-general}


In this section we prove the conjecture when $a=0$ for general $k$. 
Before proving the theorem, we provide a crucial building block of the proof, which can be viewed as a generalization of the characteristic function defined in Theorem~\ref{0b}.

\begin{proposition}\label{charGen}
If $b$ and $k$ are two integers so that each prime factor of $b$ is also a prime factor of $k$, that is, $k$ and $b$ can be decomposed as $k = p_1^{\alpha_1}\cdots p_l^{\alpha_l}$ and $b = p_1^{\beta_1}\cdots p_l^{\beta_l}$, for integers $\alpha_j > 0$ and $\beta_j \ge 0$,~(thus $[k]$ and $[b]$ can be viewed as $\mathbb{Z}_{p_1}^{\alpha_1}\times\cdots\times\mathbb{Z}_{p_l}^{\alpha_l}$ and $\mathbb{Z}_{p_1}^{\beta_1}\times\cdots\times\mathbb{Z}_{p_l}^{\beta_l}$, respectively), then there exists a linear characteristic function $\bd{q}: [k]^b \rightarrow [b]$ with the following property:

There exists $s^* \in [k]$, so that for any $\bd{x} \in [k]^b$, there is a unique index $i \in [b]$ satisfying that $\bd{q}(\bd{x} - s^* \cdot\bd{e}_i) = \bd{0}$, where $s^*\cdot \bd{e}_i$ is the vector in $[k]^b$ with all entries $0$ except that the $i$-th is $s^*$.
\end{proposition}
\begin{proof}
Define the a linear operator $\circledast: [k] \times [b] \rightarrow [b]$, as a generalization of multiplication of scalar and vector, as follows:

Suppose that $\bd{u} = (\bd{u}^1,\ldots,\bd{u}^l) \in \mathbb{Z}_{p_1}^{\alpha_1}\times\cdots\times\mathbb{Z}_{p_l}^{\alpha_l}(\cong[k])$, where each $\bd{u}^j \in \mathbb{Z}_{p_j}^{\alpha_j }$ can be represented as $\bd{u}^j = (u^j_{\alpha_j-1},\dots,u^j_{0})$. Similarly, suppose $\bd{v} = (\bd{v}^1,\ldots,\bd{v}^l)\in \mathbb{Z}_{p_1}^{\beta_1}\times\cdots\times\mathbb{Z}_{p_l}^{\beta_l}(\cong[b])$.
Define
\[
\bd{u} \; \circledast\; \bd{v} = (u^1_0\cdot \bd{v}^1, u^2_0\cdot \bd{v}^2, \ldots, u^l_0\cdot \bd{v}^l),
\]
where $u^j_0\cdot \bd{v}^j$ is the multiplication of a scalar and a vector over $\mathbb{Z}_{p_j}$.

Based on this $\circledast$ operator, define $\bd{q}: [k]^b \rightarrow [b]$ as follows:\footnote{Here we view $\bd{x}_i$ a vector in $[k] \cong \mathbb{Z}_{p_1}^{\alpha_1}\times\ldots\times\mathbb{Z}_{p_l}^{\alpha_l}$ as in the definition of $\circledast$.}
\[
   \bd{q}(\bd{x}) = \sum_{\bd{i} \in [b]\cong \mathbb{Z}_{p_1}^{\alpha_1}\times\cdots\times\mathbb{Z}_{p_l}^{\alpha_l} } \bd{x}_i \; \circledast \; \bd{i}.
\]

Now we prove that $\bd{q}(\bd{x})$ indeed has the desired property. Let $\bd{s}^* = (1,\ldots,1) \in \mathbb{Z}_{p_1}^{\alpha_1}\times\ldots\times\mathbb{Z}_{p_l}^{\alpha_l}(\cong[k])$ be the element with each entry 1\footnote{Since $s^*$ could be viewed both as an element in $[k]$ and a vector in $\mathbb{Z}_{p_1}^{\alpha_1}\times\ldots\times\mathbb{Z}_{p_l}^{\alpha_l}$, we use $s^*$ and $\bd{s}^*$ correspondingly.}, since $\circledast$ is a linear operator with addition over $\mathbb{Z}_{p_1}^{\alpha_1}\times\ldots\times\mathbb{Z}_{p_l}^{\alpha_l}$, we have that
\[
\bd{q}(\bd{x} - s^*\bd{e_i}) = \bd{q}(\bd{x}) - \bd{q}(s^*\bd{e_i}) = \bd{q}(\bd{x}) - \bd{s}^* \;\circledast\; \bd{i}.
\]
Since $\bd{s}^*$ is a vector with every entry 1, $\bd{s}^*\;\circledast\; \bd{i}  = \bd{i}$, hence the equation $\bd{q}(\bd{x} - s^*\bd{e_i})=0$ only holds for $\bd{i}=\bd{q}(\bd{x})$.
\qed
\end{proof}

Now we prove Theorem~\ref{a0Theorem} by using the above characteristic function $\bd{q}(\cdot)$.

\vskip5pt

\noindent {\it Proof of Theorem~\ref{a0Theorem}.} We only need to prove the sufficient part. Suppose that $s=\left(\frac{kb-n}{b}\right)k^{n-1}$ is a nonnegative integer. We factor $b = dr$, where $r$ is coprime with $k$ and each prime factor of $d$ is a prime factor of $k$. Thus we have $r | kb - n$, and since $d \le n$, we have $d | k^{n-1}$. By Proposition~\ref{mul} it is sufficient to prove the $r = 1$ case.

Suppose $n = tb+ h$, where $1\leq t<k$ and $0\leq h\leq b$. Similar to the approach of proving Theorem~\ref{0b}, we partition the $n$ coordinates into three groups, each of which contains $b$, $(t-1)b$ and $h$ coordinates, respectively. We represent a point in $[k]^n$ by $(\bd{x}, \bd{y}, \bd{z})$, where $\bd{x} \in [k]^{b}$, $\bd{y} \in [k]^{(t-1)b}$, and $\bd{z}\in [k]^h$. Let $\bd{q}(\bd{x})$ be the characteristic function which satisfies the condition in Proposition \ref{charGen}. Let $Q(\bd{w}) = \{(\bd{x},\bd{y},\bd{z})\in [k]^n~|~\bd{q}(\bd{x}) = \bd{w}\}$ as before. We arbitrarily choose $h$ different values $\bd{w_1},\ldots,\bd{w_h}$ from $[b]\cong \mathbb{Z}_{p_1}^{\beta_1}\times\cdots\times\mathbb{Z}_{p_l}^{\beta_l}$ (since $0\leq h\leq b$), and let
   \[
   	M = \left(Q(\bd{w_1})\cup \cdots\cup Q(\bd{w_h})\right) \cap \{(\bd{x},\bd{y}, \bd{z}) \mid \oplus(\bd{x},\bd{y}, \bd{z}) = \bd{0}\}.\footnote{We redefine $\oplus(\bd{a})$ as $a_1+a_2+\cdots+a_n$, the addition is over the group $\mathbb{Z}_{p_1}^{\alpha_1}\times\cdots\times\mathbb{Z}_{p_l}^{\alpha_l}(\cong[k])$. }
   \]

Now we describe the marking of lines. Considering the set $[k]$, $\bd{0}$ and $\bd{s}^*$ are two special elements in it. There are $k-2 \ge t-1$ elements left. Let $\tau $ be an arbitrary injective function that maps $[t-1]$ to $[k] \setminus\{\bd{0},\bd{s^*}\}$. Thus $|\tau([t-1])| = t-1$.
   \begin{enumerate}
      \item
      	 For line $(\bd{x_{-i}},\bd{y}, \bd{z})$~($i\in [b]$), if there exists a point $(\tilde{\bd{x}}, \bd{y},\bd{z})\in M$ on it, then we mark this point. Otherwise we mark the unique point $(\tilde{\bd{x}}, \bd{y},\bd{z})$ so that $\oplus(\tilde{\bd{x}}, \bd{y},\bd{z}) = \bd{s}^*$.

      \vskip5pt
      \item
      	 For line $(\bd{x},\bd{y}_{-(i,j)}, \bd{z})$~($i\in [t-1],j\in [b]$), we mark the point $(\tilde{\bd{x}}, \bd{y},\bd{z})$ on the line so that $\oplus(\tilde{\bd{x}}, \bd{y},\bd{z}) = \tau(i)$.

      \vskip5pt
      \item
      	 On line $(\bd{x}, \bd{y}, \bd{z}_{-i})~(i\in [h])$, we always mark the unique point  $(\tilde{\bd{x}}, \bd{y},\bd{z})$ with $\oplus(\tilde{\bd{x}}, \bd{y},\bd{z}) = \vec{s}^*$.
   \end{enumerate}

Next we prove that each point $(\bd{x}, \bd{y}, \bd{z})$ is marked either $b$ times or $a = 0$ time. There are three cases:
   \begin{enumerate}
      \item
		$\oplus(\bd{x}, \bd{y}, \bd{z})=0$. On lines of form $(\bd{x}, \bd{y}_{-(i,j)}, \bd{z})$ or $(\bd{x}, \bd{y}, \bd{z}_{-i})$, this point will never get marked.
	 \begin{itemize}
	 	\item if $(\bd{x}, \bd{y}, \bd{z}) \in M$. The point is marked by all lines of form $(\bd{x}_{-i}, \bd{y}, \bd{z})$. There are exactly $b$ such lines;

    \vskip3pt
	 	\item otherwise, the point is never marked.
	 \end{itemize}

    \vskip5pt
      \item
	 $\oplus(\bd{x}, \bd{y}, \bd{z})=\vec{s}^*$. By Proposition~\ref{charGen}, $\bd{q}(\bd{x}-\bd{s}^*\bd{e}_i)$ are pairwise different. On exactly $h$ lines of form $(\bd{x}_{-i}, \bd{y},\bd{z})$, there is a point in $M$. Therefore on $(b-h)$ lines of such form, $(\bd{x}, \bd{y}, \bd{z})$ will be marked. All the lines of form $(\bd{x}, \bd{y}, \bd{z}_{-i})$ will mark this point as well. On lines of form $(\bd{x}, \bd{y}_{-(i,j)}, \bd{z})$, this point will not be marked. Thus it is marked $(b-h) + h = b$ times in total.

    \vskip5pt
      \item
         $\oplus(\bd{x}, \bd{y}, \bd{z}) \in \tau([t-1])$. The point is marked on all the line of the form $(\bd{x},\bd{y}_{-(i,j)}, \bd{z})$ where $i = \tau^{-1}(\oplus(\bd{x}, \bd{y}, \bd{z}))$. Thus it is marked $b$ times.

    \vskip5pt
      \item
	 $\oplus(\bd{x}, \bd{y}, \bd{z}) \in [k]\setminus (\tau([t-1])\cup\{\bd{0},\bd{s^*}\})$. The point is never marked.
   \end{enumerate}
\qed

\section{Conclusion and Remarks}\label{section-conclusion}

In this work we investigate a conjecture of Butler and Graham on marking lines of $[k]^n$. We proved the necessary and sufficient condition of the existence of $[a,b]_k^n$ for the case when $k$ is a prime and the case when $a=0$ with general $k$.
 A natural open question is how to settle the remaining case of the conjecture, $[a,n-t]_k^n$ $(t<k)$ for general $k$.
The proof of Theorem~\ref{a,n-t} actually can be generalized to the following case when $k$ is a prime power (with an additional constrain that $(n-t-a)$ contains more prime factors than $k$).
\begin{theorem}\label{theorem-5}
$[a,n-t]_k^n$ exists if $k=p^m$, $n=t+a+r p^s$, $n-ka=ru$ and $s\geq m$.
\end{theorem}
\noindent The proof utilized the property of field $\mathbb{F}_{p^s}$. It is essentially similar to the proof of Theorem~\ref{a0Theorem} and we will put it in appendix. It is an interesting question to know whether the method here can be further generalized. The difficulty is that during this generalization, the strong symmetric property cannot be maintained.

\bibliographystyle{plain}

\bibliography{markline}

\section*{Appendix}

\begin{proof}[Theorem~\ref{theorem-5}(sketch)]
	We construct a marking of lines in $[k]^{n}$ so that each point is marked either $(n-t)$ times or $a$ times. First, similar to in the proof of Theorem~\ref{a,n-t}, we partition the $n = a+t+p^s r$ dimensions into three groups, each of which contains $p^s r$, $t$, and $a$ coordinates respectively, and represent each point in $[k]^{n}$ by $(\bd{x},\bd{y},\bd{z})$ where $\bd{x} \in [k]^{p^s r}$, $\bd{y} \in [k]^t$, and $\bd{z} \in [k]^a$. Then, we index $\bd{x}$ by a pair $(i,j)\in \mathbb{F}_p^s\times [r]$, i.e. $x_{i,j}\in [k]$ are the coordinates of $\bd{x}$~(where $i \in \mathbb{F}_p^s, j\in [r]$)\footnote{Here we use a different arithmetic system for $i$.}. We define lines $(\bd{x}_{-(i,j)},\bd{y},\bd{z})$, $(\bd{x},\bd{y}_{-i},\bd{z})$ and $(\bd{x},\bd{y},\bd{z}_{-i})$ in a similar way.

\vskip10pt
For each $\bd{x}\in [k]^{p^s r}$, we define the {\em characteristic function} $q : [k]^{p^s r} \rightarrow \mathbb{F}_{p^s}$ in a slightly different way:
\begin{equation}
    q(\bd{x}) = \sum_{i\in \mathbb{F}_{p^s}} q'\left(\sum_{j=1}^r x_{i,j}\right)\cdot i \label{charFun}
\end{equation}
$q'$ is an arbitrary function that maps $\mathbb{F}_{p^m}$ to $\mathbb{F}_{p^s}$. $q'(\sum_{j=1}^r x_{i,j})$ is an element in $\mathbb{F}_{p^s}$.

The way of marking lines and the proof of correctness is the same as that in the proof of Theorem~\ref{a,n-t}.
\qed
\end{proof}

\end{document}